\documentclass[11pt]{amsart}

\usepackage[usenames,dvipsnames]{color}
\usepackage{latexsym,xspace,enumerate,amsfonts,amsmath,amssymb,amscd,graphicx}
\usepackage{xypic,hyperref,syntonly,slashed,paralist}
\reversemarginpar
\usepackage[arrow, matrix, curve]{xy}

\newcommand{\loc}{\mr{loc}\,}
\newcommand{\T}{\mathbb{T}}

\newcommand{\N}{\mathbb{N}}
\newcommand{\Z}{\mathbb{Z}}

\newcommand{\R}{\mathbb{R}}
\newcommand{\C}{\mathbb{C}}

\newcommand{\SL}{\mr{SL}}

\newcommand{\SU}{\mr{SU}}

\newcommand{\su}{\mathfrak{su}}

\newcommand{\mf}{\mathfrak}
\newcommand{\mr}{\mathrm}

\newcommand{\mc}{\mathcal}

\newcommand{\End}{\mathop{\rm End}\nolimits}

\newcommand{\Id}{\mathop{\rm Id}\nolimits}

\renewcommand{\ker}{\mathop{\rm ker}\nolimits}

\newcommand{\ind}{\mathop{\rm ind}\nolimits}

\renewcommand{\Re}{\mathop{\rm Re}\nolimits}
\renewcommand{\Im}{\mathop{\rm Im}\nolimits}
\renewcommand{\deg}{\mathop{\rm deg}\nolimits}

\newcommand{\Tr}{\mathop{\rm Tr}\nolimits}

\renewcommand{\mod}{\mathop{\rm mod}\nolimits}
\newcommand{\del}{\partial}
\newcommand{\delb}{\bar{\partial}}

\newcommand{\note}[1]{\marginpar{\raggedright\if@twoside\ifodd\c@page\raggedleft\fi\fi\sf\scriptsize \red{RMK: #1}}}

\newcommand\red[1]{\textcolor{red}{#1}}

\newcommand{\be}{\begin{equation}}
\newcommand{\ben}{\begin{equation}\nonumber}
\newcommand{\ee}{\end{equation}}
\newcommand{\bp}{\begin{para}}
\newcommand{\ep}{\end{para}}

\newcommand{\fid}{\mathrm{fid}}

\def\D{\mathbb{D}}

\newsavebox{\dotbox}

\newtheorem{proposition}{\textbf{Proposition}}
\newtheorem{lemma}[proposition]{\textbf{Lemma}}

\newtheorem{theorem}[proposition]{\textbf{Theorem}}

\theoremstyle{definition}
\newtheorem{definition}{\textbf{Definition}}

\newtheorem*{example*}{\textbf{Example}}

\theoremstyle{remark}      
\newtheorem*{rem*}{Remark}

\newcounter{para}[section]
\newenvironment{para}[2][]{\refstepcounter{para}\noindent\ignorespaces{\bf #1\thepara. #2.} \rmfamily}{\noindent\ignorespacesafterend\bigskip}

\numberwithin{proposition}{section}
\numberwithin{definition}{section}
\begin{document}
\title[The Hitchin fibration under degenerations to noded  surfaces]{The Hitchin fibration under degenerations to noded Riemann surfaces}
\date{\today}


\author{Jan Swoboda}
\address{Mathematisches Institut der Universit\"at M\"unchen\\Theresienstra{\ss}e 39\\D--80333 M\"unchen\\ Germany}
\email{swoboda@math.lmu.de}



\maketitle

\begin{abstract}
In this note we study some analytic properties of the linearized self-duality equations on a family of smooth Riemann surfaces $\Sigma_R$ converging for $R\searrow0$ to a surface $\Sigma_0$  with a finite number of nodes. It is shown that the linearization along the fibres of the Hitchin fibration $\mathcal M_d\to\Sigma_R$  gives rise to a graph-continuous Fredholm family, the index of it being stable   when passing to the limit.  We also report on similarities and differences between properties of the Hitchin fibration in   this degeneration  and in the limit of large Higgs fields as studied in \cite{msww14}.
\end{abstract}

\section{Introduction}\label{sec:intro}

Let $\Sigma$ be a closed Riemann surface with complex structure $J$. We fix a hermitian vector bundle $(E,h)\to \Sigma$ of rank $2$ and degree $d(E)\in\Z$. 
We furthermore fix a K\"ahler metric on $\Sigma$ such that the associated K\"ahler form $\omega$ satisfies $\int_{\Sigma}\omega=2\pi$. The main object of this article is the moduli space  of  solutions $(A,\Phi)$ of Hitchin's self-duality equations 
\be
\label{hit.equ.uni}
\begin{array}{rcl}
F_A+[\Phi\wedge\Phi^*] & = & -i\mu(E)\operatorname{id}_E\omega, \\[0.4ex]
\delb_A\Phi & =& 0
\end{array}
\ee
for a unitary connection  $A\in\mc A(E,h)$ and  a {\bf Higgs field} $\Phi \in \Omega^{1,0}(\Sigma,\End(E))$. We here denote by $\mu(E)=d(E)/2$ the slope of the complex rank-$2$ vector bundle $E$.
\medskip\\
The set of smooth solutions to Eq.~\eqref{hit.equ.uni} is invariant under the action  $g^{\ast}(A,\Phi)=(g^{\ast}A,g^{-1}\Phi g)$, where $g\in \mathcal G(E,h)$ is a unitary gauge transformation. The corresponding moduli space of solutions
\begin{equation*}
\mathcal M:= \frac{\{(A,\Phi)\in\mathcal A(E,h)\times\Omega^{1,0}(\Sigma,\End(E))\mid \eqref{hit.equ.uni}\}}{\mathcal G(E,h)}
\end{equation*}
has been widely studied in the literature, beginning with Hitchin's seminal work \cite{hi87}.  We note that for the definition of $\mathcal M$ the precise choice of the  background hermitian metric $h$ is irrelevant, since any two  hermitian metrics are complex gauge equivalent.  The dependence of  \eqref{hit.equ.uni} and hence  $\mathcal M$ on the degree $d$ of the complex vector bundle $E$ is supressed from our notation.\\
\medskip\\ 
One interesting feature of $\mathcal M$ is the existence of a singular torus fibration over the complex vector space $\operatorname{QD}(\Sigma,J)$ of holomorphic quadratic differentials, called {\bf{Hitchin fibration}},
\begin{equation*}
\det\colon \mathcal M\to \operatorname{QD}(\Sigma,J),\quad [(A,\Phi)]\mapsto\det\Phi.
\end{equation*}
The   map $\det$ is well-defined since $(A,\Phi)$ is supposed to satisfy $\bar\partial_A\Phi=0$. It is holomorphic with respect to the natural complex structure on $\operatorname{QD}(\Sigma,J)$ and the one on $\mathcal M$ given by
\begin{equation*}
I(\alpha,\varphi)=(i\alpha,i\varphi)
\end{equation*}
for a pair $(\alpha,\varphi)\in \Omega^{0,1}(\Sigma,\mathfrak{sl}(E))\oplus  \Omega^{1,0}(\Sigma,\mathfrak{sl}(E))$ representing a tangent vector at some   $[(A,\Phi)]\in\mathcal M$. It is furthermore shown in \cite[\textsection 8]{hi87} that the map $\det$ is surjective and proper with preimage $\det^{-1}(q)$ (for $q$ with simple zeroes) being biholomorphically equivalent to the Prym variety of the double covering of $\Sigma$ determined by $\sqrt{q}$.\\
\medskip\\
The recent works \cite{msww14,msww15} and \cite{sw15} have been concerned with two types of degenerations of $\mathcal M$ which are of quite different nature. In the first two mentioned articles the degeneration profile of solutions $(A,\Phi)$ of  Eq.~\eqref{hit.equ.uni} has been studied in the limit of large Higgs fields, i.e.~in the case where the $L^2$-norm of $\Phi$ tends to infinity. Since the map $[(A,\Phi)]\to\|\Phi\|_{L^2(\Sigma)}^2$ on $\mathcal M$ is well-known to be proper, this analysis was used in \cite{msww14} to obtain a geometric compactification of $\mathcal M$ together with a parametrization of  charts covering  a large portion of the boundary of the compactified moduli space (corresponding to Higgs fields $\Phi$ satisfying the generic property that   $\det\Phi$ has only simple zeroes). While so-far the Riemann surface $(\Sigma,J)$ has been fixed, we consider in \cite{sw15} $\mathcal M$ as being parametrized by the complex structure $J$ and in particular focus on families of Riemann surfaces converging to a noded limit (this limit representing a boundary point of the Deligne-Mumford compactification of Teichm\"uller moduli space).  Our main result, a gluing theorem, is described in further detail below.\\
\medskip\\
The purpose of this note is to describe some geometric aspects of the Hitchin fibration under each of the above two degenerations. Concerning the first one, the results discussed here have been obtained in collaboration with Mazzeo, Wei{\ss} and Witt and will in full detail be presented in the forthcoming article \cite{msww16}, which is concerned with the large scale geometry of the natural complete $L^2$ hyperk\"ahler  metric on $\mathcal M$. Concerning the second degeneration, we here study the linearization of the self-duality equations \eqref{hit.equ.uni}, which gives rise to a family of (nonuniformly) elliptic operators degenerating to a so-called $b$-operator in the limit of a noded surface. This situation is not unlike to the one studied by Mazzeo and the author in \cite{masw16}, where the existence of a complete polyhomogeneous expansion of the Weil-Petersson metric on the Riemann moduli space of conformal structures was shown.  Further results concerning similar geometric aspects of this latter degeneration will be presented elsewhere.

\bigskip

\centerline{\textbf{Acknowledgements}}

The author  would like to thank Hartmut Wei{\ss} for   a number of valuable comments and useful discussions.

\smallskip

\section{Preliminaries on the moduli space of solutions to Hitchin's self-duality equations}

%
\subsection{Hitchin's  self-duality equations}\label{sect:hitchineq}

We recall some basic facts concerning Higgs bundles and the self-duality equations  \eqref{hit.equ.uni}, referring to   \cite{msww14} for  further details. First we note that the second equation  in \eqref{hit.equ.uni} implies that any solution $(A,\Phi)$ determines a {\bf Higgs bundle} $(\delb,\Phi)$, i.e.\ a holomorphic structure $\delb=\delb_A$ on $E$ for which $\Phi$ is holomorphic: $\Phi\in   H^0(\Sigma,\End(E)\otimes K_{\Sigma})$, $K_{\Sigma}\cong T_{\C}^{\ast}\Sigma$ denoting  the canonical line bundle of $\Sigma$.   Conversely, given a Higgs bundle $(\delb,\Phi)$, the operator $\delb$ can be augmented to a unitary connection $A$ such that the first Hitchin equation holds provided $(\delb,\Phi)$ is {\bf stable}. Stability here means that $\mu(F)<\mu(E)$ for any  $\Phi$-invariant holomorphic line-bundle $F$, that is, $\Phi(F)\subset F\otimes K_{\Sigma}$.
\medskip\\
In the following notation such as $\mathfrak u(E)$ refers to the bundle of endomorphisms of $E$ which are hermitian with respect to the fixed hermitian metric $h$ on $E$.  This bundle splits as $\mf{su}(E)\oplus i\underline{\R}$,   the splitting being induced by  the Lie algebra splitting $\mf{u}(2)\cong \mf{su}(2)\oplus \mf{u}(1)$ into trace-free and pure trace summands. Consequently, the curvature $F_A$ of a unitary connection  $A$ decomposes as
$$
F_A = F_A^\perp+\frac{1}{2}\Tr(F_A)\otimes\operatorname{id}_E,
$$
where $F_A^\perp\in\Omega^2(\Sigma,\mf{su}(E))$ is its   trace-free part  and $\frac{1}{2}\Tr(F_A) \otimes\operatorname{id}_E$ is the   pure trace  or   central part, see e.g.~\cite{po92}. Note that $\Tr(F_A) \in \Omega^2(i\underline{\R})$ equals the curvature of the induced connection on the determinant line bundle $\det E$. From now on, we fix a background connection $A_0\in\mc A(E,h)$   and consider only those connections $A$ which induce the same connection on $\det E$ as $A_0$ does. Equivalently, such a connection $A$ is of the form $A=A_0 + \alpha$ where $\alpha\in\Omega^1(\Omega,\mf{su}(E))$, i.e.~$A$ is trace-free ``relative'' to $A_0$. We choose the  unitary background connection $A_0$ on $E$ such that $\Tr F_{A_0}=-i\deg(E)\omega$. This decomposition allows us to  replace  Eq.~\eqref{hit.equ.uni} with the slightly easier  system of equations 
\be\label{eq:hitequ}
\begin{array}{rcl}
F_A^\perp+[\Phi\wedge\Phi^*] & = & 0, \\[0.4ex]
\delb_A\Phi & =& 0
\end{array}
\ee
for $A$ trace-free relative to $A_0$ and a trace-free Higgs field $\Phi\in \Omega^{1,0}(\Sigma,\End_0(E))$.    The relevant  groups of gauge transformations in this fixed determinant case are   $\mc G^c=\Gamma(\Sigma,\SL(E))$ and $\mc G=\mathcal G(E,h)=\Gamma(\Sigma,\SU(E,h))$, the former being the complexification of the latter.

\subsection{The limit of large Higgs fields}\label{subsec:limitlargeHiggs}

We discuss the results obtained in \cite{msww14} concerning the degeneration of solutions $(A,\Phi)$ of Eq.~\eqref{eq:hitequ} as $\|\Phi\|_{L^2(\Sigma)}\to\infty$. For this purpose we introduce the {\bf{rescaled self-duality equations}}
\be\label{eq:hitequresc}
\begin{array}{rcl}
F_A^\perp+t^2[\Phi\wedge\Phi^*] & = & 0, \\[0.4ex]
\delb_A\Phi & =& 0
\end{array}
\ee
for some parameter $t>0$.

%
\subsubsection*{Existence of limiting configurations}\label{limi.prym}
We describe the  following local model for degenerations of solutions to Eq.~\eqref{eq:hitequresc}  as $t\to\infty$.  On $\C$ we consider the pair  $(A_t^{\fid}, \Phi_t^{\fid})$ given in coordinates $z=re^{i\theta}$ by
\begin{equation*}
A_t^{\fid}=\left(\frac{1}{8}+\frac{1}{2}\frac{\partial h_t}{\partial r}\right)\begin{pmatrix}1&0\\0&-1\end{pmatrix}\left(\frac{dz}{z}-\frac{d\bar z}{\bar z}\right),\quad 
\Phi_t^{\fid}=\begin{pmatrix}0&r^{\frac{1}{2}}e^{h_t}\\ r^{\frac{1}{2}}e^{-i\theta}e^{-h_t}&0\end{pmatrix}\, dz.
\end{equation*}
It automatically satisfies the second equation in \eqref{eq:hitequresc}. The  
 function $h_t\colon(0,\infty)\to\R$ is determined in such a way that also the first equation holds. The existence and precise  properties  of the $1$-parameter family of  functions  $h_t$ have been discussed in \cite{msww14,msww15}. We only mention here that it is smooth on $(0,\infty)$, has a logarithmic pole at $r=0$ and satisfies $\lim_{t\to\infty} h_t=0$ uniformly on compact subsets of $(0,\infty)$. Note that in this ansatz the determinant $\det \Phi_t^{\fid}=-z\, dz^2$ is independent of $t$ and has a simple zero at $z=0$.  We also point out that for $t\to\infty$ the family of smooth solutions $(A_t^{\fid}, \Phi_t^{\fid})$ has a limit which is singular in $z=0$ and which satisfies the decoupled self-duality equations \eqref{eq:decoupledsd} below.
Here and in the following we denote
\begin{equation*}
\Sigma^{\times}:=\Sigma\setminus q^{-1}(0)
\end{equation*}
for some given Higgs field $\Phi$ where $q=\det\Phi$. We call the Higgs field {\bf{simple}} if $q$ has only simple zeroes, the number of which then equals $4(\gamma-1)$.

\begin{definition}\label{lim.conf.equ}
A {\bf limiting configuration} is a   pair 
$(A_\infty,\Phi_\infty)$ such that $\Phi_{\infty}$ is simple and  which satisfies the decoupled self-duality equations 
\begin{equation}\label{eq:decoupledsd}
F_{A_\infty}^\perp=0,\quad[\Phi_\infty\wedge\Phi_\infty^*]=0,\quad\delb_{A_\infty}\Phi_\infty=0
\end{equation}
on $\Sigma^\times$, and which furthermore agrees with $(A_{\infty}^{\fid}, \Phi_{\infty}^{\fid})$ near each point of $(\det\Phi)^{-1}(0)$ with respect to some unitary frame for $E$ and local holomorphic coordinate system such that $\det\Phi=-z\,dz^2$. 
\end{definition}

The main result for limiting configurations is the following.

\begin{theorem}\label{ex.out.sol}
Let $(E,h,\Phi)$ be a hermitian Higgs bundle  with simple Higgs field. Let $A_h$ denote the Chern connection associated with $(E,h)$. Then in the complex gauge orbit of $(A_h,\Phi)$ over 
$\Sigma\setminus(\det\Phi)^{-1}(0)$ there exists a limiting configuration $(A_\infty,\Phi_\infty)=g_{\infty}^{\ast}(A_h,\Phi)$. It is unique up to a unitary gauge transformation. The limiting complex gauge transformation $g_\infty$ is singular in the points of $(\det\Phi)^{-1}(0)$, near which it takes the form 
\begin{equation*}
g_\infty=\left(\begin{array}{cc}|z|^{-\tfrac{1}{4}}&0\\0&|z|^{\tfrac{1}{4}}\end{array}\right)
\end{equation*}
up to multiplication with a smooth unitary gauge transformation on $\Sigma$. 
\end{theorem}
 
This result was proved in \cite{msww14}. It relies on  the Fredholm theory of conic elliptic operators, in this case for the twisted Laplacian $\Delta_{A_{\infty}}$.

\subsubsection*{Desingularization by gluing}
We finally describe a partial converse, also shown in  \cite{msww14}, to Theorem \ref{ex.out.sol}. It entails a global version of the observation that  the family of smooth fiducial solutions $(A_t^{\fid}, \Phi_t^{\fid})$ desingularizes the limiting fiducial solution $(A_{\infty}^{\fid}, \Phi_{\infty}^{\fid})$. These solutions are obtained by gluing  $(A_\infty,\Phi_\infty)$, which on the complement of some neighbourhood of  $(\det\Phi)^{-1}(0)$ is a bounded solution to  Eq.~\eqref{eq:hitequresc} for any  $t$,  to the 
 model solution  $(A_t^{\fid},\Phi_t^{\fid})$ for some large but finite $t$.

\begin{theorem}\label{gluing_theorem}
Suppose that $(A_{\infty}, \Phi_{\infty})$ is a  limiting configuration where $\Phi_{\infty}$ is a simple Higgs field. Then there exists a family of smooth solutions $(A_t, \Phi_t)$ to the rescaled self-duality  equations \eqref{eq:hitequresc}  with $A_t \to A_{\infty}$ and $\Phi_t \to \Phi_{\infty}$  in $C^{\infty}_{\loc}$ at exponential rate in $t$   on the complement of $(\det \Phi_{\infty})^{-1}(0)$. 
\end{theorem}

\subsection{The limit under degeneration to a noded surface}\label{subsec:limitnodedsurface}

Following   \cite{sw15} we introduce  the setup   for our study of the Hitchin fibration on a family of Riemann surfaces degenerating to a  surface with one or more nodes.

\subsubsection*{Plumbing construction} 

We briefly recall the conformal plumbing construction for Riemann surfaces. A {\bf{Riemann surface with nodes}} is  a closed  one-dimensional complex manifold  with singularities $\Sigma_0$ where each point has a neighbourhood complex isomorphic to either a disk $\{|z|<\epsilon\}$ or to $U=\{zw=0\mid \left|z\right|,\left|w\right|<\epsilon\}$, in which case the point is called a node. A Riemann surface with  nodes arises from an unnoded surface by pinching of one or more  simply closed curves. Conversely, the effect of  the so-called conformal plumbing construction  is that it opens up a node by replacing the neighbourhood $U$ by $\{zw=t\mid t\in\C,\left|z\right|,\left|w\right|<\epsilon\}$. To describe this construction in more detail, let $(\Sigma_0,z,p)$ be a Riemann surface of genus $\gamma\geq2$ with conformal coordinate $z$ and a single node at $p$. Let $t\in\C\setminus\{0\}$ be fixed with $\left|t\right|$ sufficiently small. We then  define a smooth Riemann surface $\Sigma_t$ by removing the disjoint disks $D_t=\{\left|z\right|<\left|t\right|,\left|w\right|<\left|t\right|\}\subseteq U$ from $\Sigma_0$ and passing to the quotient space $\Sigma_t=(\Sigma_0\setminus D_t)/_{zw=t}$, which  is a Riemann surface of the same genus as $\Sigma_0$.   In the following we allow for Riemann surfaces with a finite number of nodes, the set of which we denote by $\mathfrak p=\{p_1,\ldots, p_k\}\subset \Sigma$. The value of $t$ may be different at different nodes. We let  $R:=\max_{p\in\mathfrak p}\left|t(p)\right|^2$ be  the maximum of the squares of these   absolute values. To deal with the case of multiple nodes in an efficient way we make the {\bf{convention}} that in the notation $\Sigma_R$ the dependence of the parameter  $t\in\C$ on the point  $p\in\mathfrak p$ is suppressed.\\
\medskip\\   
Let $\rho=\left|t\right|<1$ and consider the annuli
\begin{equation}\label{eq:annuli}
R_{\rho}^+=\{z\in\C\mid \rho\leq\left|z\right|\leq 1\}\qquad\textrm{and}\qquad  R_{\rho}^-=\{w\in\C\mid \rho\leq\left|w\right|\leq 1\}.
\end{equation}
The above identification of $R_{\rho}^+$ and $R_{\rho}^-$ along their inner boundary circles $\{\left|z\right|=\rho\}$ and $\{\left|w\right|=\rho\}$ yields a smooth cylinder $C_{t}$.\\
\medskip\\   
As before, we let $\operatorname{QD}(\Sigma)=H^0(\Sigma,K_{\Sigma}^2)$ denote  the $\C$-vector space of holomorphic quadratic differentials on  $\Sigma$. On a noded Riemann surface we will allow for quadratic differentials meromorphic  with poles of order at most $2$ at points in the subset $\mathfrak p\subset \Sigma$ of nodes.  In this case,  the corresponding $\C$-vector space of meromorphic quadratic differentials is denoted by $\operatorname{QD}_{-2}(\Sigma)$.


\begin{figure}   
\label{fig:1}                                                   %
\includegraphics[width=1.0\textwidth]{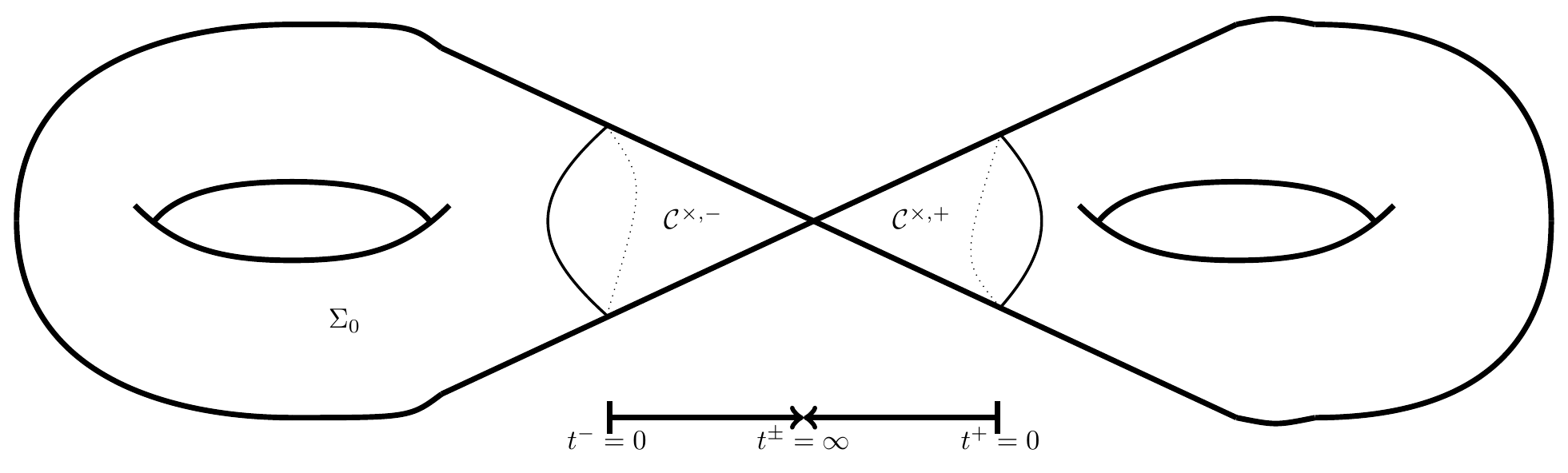}                         \caption{Degenerate Riemann surface $\Sigma_0$ with one node $p\in\mathfrak p$, which separates the two half-infinite cylinders $\mathcal C_0^{\pm}$.}
\end{figure}

%
%
\subsubsection*{The local model} 
 
We briefly describe how to extend the setup of the self-duality equations to the case where the underlying manifold is a noded Riemann surface $\Sigma_0$. The complex rank-$2$ vector bundle in this situation is  supposed to be a {\bf{cylindrical vector bundle}} as discussed e.g.~in \cite{clm96,ni02}. By this we mean that a pair of   cylindrical coordinates $(\tau^{\pm},\theta^{\pm})$ with
\begin{equation*}
\tau^{\pm}=\left|\log r^{\pm}\right|
\end{equation*}
is chosen, one for each of the two connected components $\mathcal C_0^{\pm}$  of the punctured neighbourhood $\mathcal C_0$ of $p\in\mathfrak p$. 
We then fix a smooth hermitian metric $h$ on $E$     in such a way that    its restriction to $E|_{\mathcal C_0^{\pm}}$ is invariant under pullback by translations in the $\tau^{\pm}$-direction. We furthermore require that $h$ is invariant under pullback via the isometric involution $(\tau^{\pm},\theta^{\pm})\mapsto(\tau^{\mp},\operatorname{arg}t-\theta^{\mp})$ interchanging the two half-infinite cylinders $\mathcal C^+$ and $\mathcal C^-$. The pair $(E,h)$ induces a hermitian vector bundle on each surface $\Sigma_R$ by restriction, which by the assumptions on $h$ extends smoothly over the cut-locus $\left|z\right|= \left|w\right|=\rho$, cf.~Figure \ref{fig:2}.\\ 
\medskip\\
We next fix constants $\alpha\in\R$ and $C\in\C$. Then the pair
\begin{equation}\label{eq:modsol}
A^{\mod}=\begin{pmatrix}\alpha&0\\0&-\alpha\end{pmatrix}\left(\frac{dz}{z}-\frac{d\bar z}{\bar z}\right),\quad \Phi^{\mod}=\begin{pmatrix}C&0\\0&-C\end{pmatrix}\frac{dz}{z}
\end{equation}
provides a  solution on $\C$ to the self-duality equations \eqref{hit.equ.uni}, which we call {\bf{model solution to parameters}} $(\alpha,C)$. It is smooth outside the origin and has a logarithmic (first-order) singularity in $z=0$, provided that $\alpha$ and $C$ do not both vanish. It furthermore restricts to a smooth solution on each of the annuli $R_{\rho}^{\pm}$ defined in \eqref{eq:annuli}. For constants $t\in\C$ and $\rho=\left| t\right|$ such that $0<\rho<1$ let $\mathcal C_t$ denote the complex cylinder obtained from gluing the two annuli $R_{\rho}^-$ and $R_{\rho}^+$. Since
\begin{equation*}
\frac{dz}{z}=-\frac{dw}{w}
\end{equation*} 
the two model solutions $(A_+^{\mod},\Phi_+^{\mod})$  to parameters $(\alpha,C)$ over $R_{\rho}^+$ and  $(A_-^{\mod},\Phi_-^{\mod})$ to parameters $(-\alpha,-C)$ over $R_{\rho}^-$ glue to a smooth solution $(A^{\mod},\Phi^{\mod})$ on $\mathcal C_t$, again called model solution to parameters $(\alpha,C)$. In the following it is always assumed that $\alpha>0$.\\
\medskip\\
We now impose the following   {\bf{assumptions}}.
\begin{itemize}
\item[(A1)]
For each $p\in\mathfrak p$, the constant $C_p=C_{p,+}$ is nonzero. 
\item[(A2)]
For each $p\in\mathfrak p$, the constants $C_{p,+}$ and $C_{p,-}$ satisfy $C_{p,+}=-C_{p,-}=C_p$.  
\item[(A3)]
The meromorphic quadratic differential $q=\det\Phi$ has only simple zeroes. 
\end{itemize}
The relevance of the model solutions is that any solution of the self-duality equations    with  logarithmic singularities in $\mathfrak p$ is exponentially close (with respect to the above cylindrical coordinates)  to some model solution, a fact which is due to  Biquard-Boalch \cite[Lemma 5.3]{bibo04}.

\begin{figure}   
\includegraphics[width=1.0\textwidth]{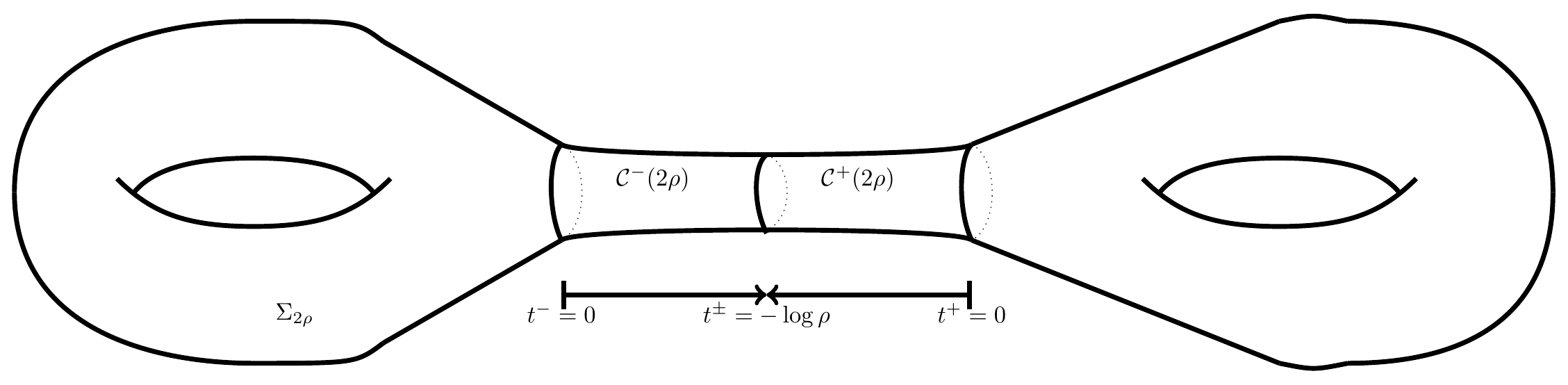}                         
\caption{Setup for the gluing theorem (Thm.~\ref{thm:mainthm}).}
\label{fig:2}       
\end{figure} 
The main result shown in \cite{sw15} is the following gluing theorem.

\begin{theorem}\label{thm:mainthm}
Let $(\Sigma,J_0)$ be a  Riemann surface with nodes in a finite set of points $\mathfrak p\subset\Sigma$. Let   $(A_0,\Phi_0)$ be a solution of the self-duality equations  with   logarithmic singularities in $\mathfrak p$, thus representing a point in $\mathcal M(\Sigma,J_0)$. Suppose that  $(A_0,\Phi_0)$ satisfies the assumptions (A1--A3) stated above.  Let  $(\Sigma,J_i)$ be a sequence of smooth  Riemann surfaces converging uniformly to $(\Sigma,J_0)$. Then, for every sufficiently large $i\in\N$, there exists  a smooth solution $(A_i,\Phi_i)$ of Eq.~\eqref{eq:hitequ} on $(\Sigma,J_i)$ such that $(A_i,\Phi_i)\to (A_0,\Phi_0)$ as $i\to\infty$ uniformly on compact subsets of $\Sigma\setminus\mathfrak p$.
\end{theorem}

\section{The Hitchin fibration in the limit of large Higgs fields}\label{sec:fibrationlargeHiggs}

Throughout this section it is supposed that the above assumption (A3) holds. Recall the Definition \ref{lim.conf.equ}  of limiting configurations. Our first observation is that   the second component $\Phi_\infty$ of a limiting configuration is completely determined up to  a unitary gauge transformation by the  holomorphic quadratic differential $q$. This is a consequence of the standard fact that any normal endomorphism is diagonalizable by  some $g\in \SU(n)$.
Now consider the space of unitary connections solving 
\begin{equation}
\bar\partial_A \Phi_\infty = 0, \qquad F_A^\perp = 0;
\label{deflcc}
\end{equation}
the gauge freedom is the  stabilizer of $\Phi_\infty$ in $\Gamma(\Sigma^\times,\SU(E))$, i.e.~the group of unitary gauge transformations of the complex line bundle
\begin{equation}\label{eq:complexline}
L_{\Phi_{\infty}}:=\{\gamma\in\End(E)\mid [\Phi_{\infty}\wedge\gamma]=0\}
\end{equation}
over $\Sigma^{\times}$. Fix a base solution $A_\infty$ of Eq.~\eqref{deflcc} and write 
$A = A_\infty + \alpha$, where $\alpha \in \Omega^1(\Sigma^{\times},\mf{su}(E))$. The first equation   gives that 
\begin{equation*}
[\alpha^{0,1} \wedge \Phi_\infty]=[ \alpha \wedge \Phi_\infty] = 0,
\end{equation*}
so $\alpha$ takes values in the real line bundle 
$L_{\Phi_\infty}^{\R}:=L_{\Phi_\infty} \cap \su(E)$. This implies in particular that $[\alpha \wedge \alpha]=0$.  From the second equation 
of \eqref{deflcc} we obtain $d_{A_\infty} \alpha = 0$, hence the ungauged deformation space at $(A_\infty, \Phi_\infty)$ can be identified with
$$
Z^1(\Sigma^\times;L_{\Phi_\infty}^{\R}):=\{\alpha\in\Omega^1(\Sigma^\times,L_{\Phi_\infty}^{\R})\mid d_{A_\infty}\alpha=0\}.
$$
Next consider the subgroup $\operatorname{Stab}_{\Phi_\infty}$ of unitary gauge transformations which fix $\Phi_\infty$. 
If $g \in \operatorname{Stab}_{\Phi_\infty}$ is of the form $g = \exp(\gamma)$ for some  $\gamma\in\Omega^0(\Sigma^\times,L_{\Phi_\infty}^{\R})$, 
then $g$ acts on $\alpha\in\Omega^1(\Sigma^\times,L_{\Phi_\infty}^{\R})$ by
$$
\alpha^g=g^{-1} \alpha  g+g^{-1}(d_{A_\infty}g)=\alpha+d_{A_\infty}\gamma.
$$
We here use that $L_{\Phi_\infty}^{\R}$ is a parallel line subbundle of $\mf{su}(E)$ with respect to $A_\infty$, so $g^{-1} \alpha  g=\alpha$ and 
$d_{A_\infty}\exp(\gamma)=\exp(\gamma)d_{A_\infty}\gamma$. Hence the gauged deformation space is 
\begin{equation*}
H^1(\Sigma^\times;L_{\Phi_\infty}^{\R})=\frac{Z^1(\Sigma^\times;L_{\Phi_\infty}^{\R})}{B^1(\Sigma^\times;L_{\Phi_\infty}^{\R})},
\end{equation*}
where 
\begin{equation*}
B^1(\Sigma^\times;L_{\Phi_\infty}^{\R}):=\{d_{A_\infty}\gamma\mid\gamma\in\Omega^0(\Sigma^\times, L_{\Phi_\infty}^{\R})\}.
\end{equation*}

\begin{lemma}[cf.~\cite{msww14}]\label{lem:indexlimtorus}
Suppose the assumption (A3) holds. Then
\[
\dim_\R H^1(\Sigma^\times; L_{\Phi_\infty}^{\R}) = 6(\gamma-1),
\]
where $\gamma$ is the genus of $\Sigma$.
\end{lemma}

\begin{proof}
Working either with $\Sigma^\times$ or the homotopy equivalent space $M = \Sigma \setminus B_\varepsilon(\mathfrak p)$ for some sufficiently small $\varepsilon>0$  (so $\del M$ is a union of $k$ circles, 
$k = |\mathfrak p|$), we note the following. First, there are no nontrivial parallel sections since $L_{\Phi_\infty}^{\R}$ is twisted near each 
$p_i$, so $H^0(\Sigma^\times;L_{\Phi_\infty}^{\R})=0$; by Poincar\'e duality, $H^2( \Sigma^\times; L_{\Phi_\infty}^{\R}) = H^0(M,\partial M;L_{\Phi_\infty}^{\R})=0$ 
as well. Recall also that since $L_{\Phi_\infty}^{\R}$ is a flat real line bundle, i.e., a local system of rank $1$, it has Euler-Poincar\'e characteristic 
\[
\chi(\Sigma^\times;L_{\Phi_\infty}^{\R}) = \chi(\Sigma^\times) = 2-2\gamma-k.
\]
These facts together give that
\[
\dim_\R H^1(\Sigma^\times; L_{\Phi_\infty}^{\R}) =k + 2\gamma-2 = 4\gamma-4 + 2\gamma-2 = 6\gamma-6,
\]
as claimed.
\end{proof}

We now turn to a discussion of the asymptotic behaviour of the natural $L^2$ Riemannian metric $G$ on $\mathcal M$ in the limit of large Higgs fields. By definition
\begin{equation}\label{eq:defl2metric}
G_x((\alpha_1,\varphi_1),(\alpha_2,\varphi_2))=2\Re\int_{\Sigma}\Tr(\alpha_1^{\ast}\wedge \alpha_2+\varphi_1\wedge \varphi_2^{\ast})
\end{equation}
for $x=[(A,\Phi)]\in\mathcal M$ and a pair $(\alpha_j,\varphi_j)\in \Omega^{0,1}(\Sigma,\mathfrak{sl}(E))\oplus \Omega^{1,0}(\Sigma,\mathfrak{sl}(E))$ in unitary gauge, representing a pair of tangent vectors of $\mathcal M$ in $x$. The metric $G$ has been  introduced by Hitchin in \cite{hi87}; remarkably, it  is a complete (in the case where  $d=\deg E$ is odd) hyperk\"ahler metric.   The definition \eqref{eq:defl2metric}  of $G$  cannot be extended to points   $x=(A_{\infty},\Phi_{\infty})$, the limiting configurations of   \textsection \ref{subsec:limitlargeHiggs}. Namely, it  turns out that infinitesimal variations of limiting configurations in directions transversally to the Hitchin fibration  fail to have finite $L^2$ norm. In contrast, the restriction of the metric $G$ to  its fibres extends to the limit, in a way which we describe next.\\
\medskip\\
For a simple holomorphic quadratic differential $q\in \operatorname{QD}(\Sigma)$ we let
\begin{equation*}
\T_q:=\{[(A_{\infty},\Phi_{\infty})]\mid \eqref{eq:decoupledsd}\;\textrm{and}\; \det\Phi_{\infty}=q\}
\end{equation*}
denote  the associated torus of  limiting configurations, which by Lemma \ref{lem:indexlimtorus} has  real dimension $6(\gamma-1)$. We then set  (supressing the dependence from the base point $x=[(A_{\infty},\Phi_{\infty})]$ from the notation)
\begin{equation}\label{eq:l2torusmetric}
G_q(\alpha_1,\alpha_2)=2\Re\int_{\Sigma}\Tr(\alpha_1^{\ast}\wedge\alpha_2)
\end{equation}
for
\begin{equation*}
\alpha_1,\alpha_2\in T_x\T_q=\{\alpha\in\Omega^1(\Sigma^{\times},L_{\Phi_{\infty}}^{\R})\mid \alpha\in\ker(d_{A_{\infty}}\oplus  d_{A_{\infty}}^{\ast})\}.
\end{equation*}
Assuming finiteness of the integral on the right-hand side of   \eqref{eq:l2torusmetric} for a moment,  Lemma \ref{lem:indexlimtorus} can now  easily be used to deduce flatness of the  Riemannian metric $G_q$. Namely, we observe  that for any $x\in\T_q$ a tangent frame of $T_x\T_q$ induces a flat   local coordinate system on some neighbourhood of $x$. Indeed, let $(\beta,0)\in T_{x}\T_q$ and put $B_{\infty}:= A_{\infty}+\beta$. Then it follows that the unitary connection $B_{\infty}$ satisfies
\begin{equation*}
F_{B_{\infty}}=F_{A_{\infty}}+d_{A_{\infty}}\beta+\frac{1}{2}[\beta\wedge\beta]=F_{A_{\infty}}=0,
\end{equation*}
since $[\beta\wedge\beta]=0$ for any $1$-form with values in $L_{\Phi_{\infty}}$ and $d_{A_{\infty}}\beta=0$ by assumption. Therefore $(B_{\infty},\Phi_{\infty})$ represents again a point in $\T_q$.   It follows that the Riemannian metric $G_q$ has constant coefficients with respect to the above local coordinate system. Therefore its curvature vanishes.\\
\medskip\\
The outlined result  has been obtained in  collaboration with R.~Mazzeo, H.~Wei{\ss} and F.~Witt as part of a much broader study of the asymptotic geometry of the moduli space $\mathcal M_d$, cf.~the forthcoming article \cite{msww16}. There it is shown that  the restriction of the  metric $G$   to the fibre over $tq$ of the Hitchin fibration converges to $G_q$   as $t\to\infty$. This also explains finiteness of the integral in \eqref{eq:l2torusmetric}.   Concerning the structure of the $L^2$ metric on $\mathcal M_d$ itself it is proven that it is asymptotically close to the well-studied semiflat hyperk\"ahler metric $G_{\operatorname{sf}}$, an incomplete Riemannian metric which is defined only on the region $\mathcal M^{\ast}$ comprised by the simpe Higgs fields. This metric stems     from the data of an algebraic completely integrable system determined by restricting the Hitchin fibration to  $\mathcal M^{\ast}$. The term `semiflat'  here refers to the fact that the fibres $\det^{-1}(q)$ are exactly flat with respect to $G_{\operatorname{sf}}$. Moreover, this metric induces on the base $\operatorname{QD}(\Sigma)$ a K\"ahler metric with further interesting properties, a so-called special K\"ahler metric as studied e.g.~by Freed \cite{fr99}.    Our  investigation here is guided by the  conjectural picture due to  Gaiotto, Moore and Neitzke \cite{gmn10,gmn13} which describes the $L^2$ metric $G$ as a perturbation series off the semi-flat metric $G_{\operatorname{sf}}$.

\section{The Hitchin fibration under degenerations of the Riemann surface}\label{sec:fibrationdegsurface}

The aim of this section is to carry part of   the preceding discussion over to the case of a family of Riemann surfaces degenerating to a noded limit.
To start with, we recall that the linearization of the self-duality equations at a solution $(A,\Phi)$ gives rise to the operator
\begin{equation}\label{eq:linoperator}
D_{(A,\Phi)}\colon (\alpha,\varphi)\mapsto\begin{pmatrix}d_A\alpha+[\Phi\wedge\varphi^{\ast}]+[\Phi^{\ast}\wedge\varphi]\\
\bar\partial_A\varphi+[\alpha^{0,1}\wedge\Phi]\end{pmatrix}.
\end{equation}
Elements in the nullspace of $D_{(A,\Phi)}$ represent tangent vectors of $\mathcal M$ at $x=[(A,\Phi)]$ up to unitary gauge. We are here only interested in the directions tangential at $x$ to the fibres of the Hitchin fibration, the space of which we denote by $\mathcal V_x$. It follows from \eqref{eq:linoperator} that the space
\begin{equation*}
\mathcal W_x:=\{\alpha\in\Omega^1(\Sigma,\mathfrak{su}(E))\mid d_A\alpha=0,  d_A^{\ast}\alpha=0, [\alpha^{0,1}\wedge\Phi]=0\},
\end{equation*}
the equation $d_A^{\ast}\alpha=0$ constituting a gauge-fixing condition, represents    a subspace of $\mathcal V_x$.  By comparing dimensions, we show below that over a smooth surface $\Sigma$ both spaces do in fact coincide. Because of the decompositions
\begin{equation*}
d_A=\partial_A+\bar\partial_A\qquad\textrm{and}\qquad d_A^{\ast}=-\ast \partial_A\ast-\ast\bar\partial_A\ast
\end{equation*}
we see that the nullspaces of the operators $d_A+d_A^{\ast}$ and $\bar\partial_A\colon\Omega^{1,0}(\Sigma,\mathfrak{sl}(E))\to \Omega^{1,1}(\Sigma,\mathfrak{sl}(E))$ are in bijection to each other. It therefore suffices to consider instead of $d_A+d_A^{\ast}$ the simpler operator $\bar\partial_A$. Indeed, writing the connection $A$ with respect to a local unitary frame over $U\subset\Sigma$  as $A=\beta\, d\bar z-\beta^{\ast}\, dz$ for some $\beta\in C^{\infty}(U,\mathfrak{sl}(2,\C))$ and using  $(\alpha^{1,0})^{\ast}=-\alpha^{0,1}$ it follows that
\begin{equation*}
d_A\alpha=\bar\partial \alpha^{1,0}+[\beta\, d\bar z\wedge \alpha^{1,0}] +\partial\alpha^{0,1}-[\beta^{\ast}\,dz\wedge \alpha^{0,1}]=2\Re\bar\partial_A\alpha^{1,0},
\end{equation*}
and similarly
\begin{eqnarray*}
-\ast d_A^{\ast}\alpha&=& \bar\partial \ast \alpha^{1,0}+[\beta\, d\bar z\wedge \ast\alpha^{1,0}]+\partial\ast\alpha^{0,1}-[\beta^{\ast}\,dz\wedge \ast\alpha^{0,1}]\\
&=& \bar\partial (-i \alpha^{1,0})+[\beta\, d\bar z\wedge (-i\alpha^{1,0})]+\partial (-i\alpha^{0,1})-[\beta^{\ast}\,dz\wedge  (-i\alpha^{0,1})]\\
&=&2\Im \bar\partial_A\alpha^{1,0},
\end{eqnarray*}
and therefore $\bar\partial_A\alpha^{1,0}$ determines $(d_A+d_A^{\ast})\alpha$ and vice versa.\\
\medskip\\
Although it is at present not clear whether the equality $\mathcal W_x=\mathcal V_x$ continuous to hold   for the noded limit $\Sigma_0$ of the smooth family $\Sigma_R$ of Riemann surfaces (we here use the notation introduced in \textsection \ref{subsec:limitnodedsurface}), it    motivates our study of the behaviour of the family of operators  $\bar\partial_A$   in the limit $R\searrow0$. For a Higgs field $\Phi$ we set as before $\Sigma_R^{\times}=\Sigma_R\setminus(\det\Phi)^{-1}(0)$. We also recall   Definition \ref{lim.conf.equ} of a limiting Higgs field $\Phi_{\infty}$ on $\Sigma_R^{\times}$, i.e.~a Higgs field which satisfies $[\Phi_{\infty}\wedge\Phi_{\infty}^{\ast}]=0$. We keep the above assumption (A3) that all the Higgs fields we consider have simple determinants.

\begin{proposition}\label{prop:isomorlinebundles}
Let   $\Phi$ be a smooth  Higgs field on $\Sigma_R$ such that $\det\Phi=q$.  Then there exists  a limiting Higgs field $\Phi_{\infty}$     of the same determinant $q$, and moreover,  for any such  $\Phi_{\infty}$ the line bundles $L_{\Phi}$ and $L_{\Phi_{\infty}}$   (cf.~\eqref{eq:complexline}) are isomorphic as complex vector bundles over $\Sigma_R^{\times}$. 
\end{proposition}

\begin{proof}
Since by assumption  $q=\det \Phi$ has only simple zeroes, there exists a limiting Higgs field $\Phi_{\infty}$ as introduced in \textsection \ref{sec:fibrationlargeHiggs} of the same determinant $q$. Furthermore, as shown in \cite[Lemma 4.2]{msww14} one can choose a smooth section $g\in \Gamma(\Sigma_R^{\times},\SL(E))$ such that  $\Phi=g^{-1}\Phi_{\infty} g$. Since the  Lie group $\SL(2,\C)$ is  homotopy-equivalent to the simply-connected manifold $S^3$ and $\Sigma_R^{\times}$ retracts onto a bouquet of circles there are no obstructions to  the existence of a smooth path $g_t\colon t\mapsto  g_t\in \Gamma(\Sigma_R^{\times},\SL(E))$ satisfying $g_0=\Id$ and $g_1=g$. Therefore the complex  line bundles $L_{\Phi}$ and $L_{\Phi_{\infty}}$ are isomorphic. 
\end{proof}

\begin{proposition}\label{prop:linebundlepar}
Let $(A,\Phi)$ be a smooth solution of the self-duality equations \eqref{eq:hitequ} on $\Sigma_R$, $R>0$. Then the line bundle $L_{\Phi}$ over $\Sigma_R^{\times}$ is parallel with respect to the connection  $A$. 
\end{proposition}

\begin{proof}
Let $\gamma\in\Omega^0(\Sigma_R^{\times},L_{\Phi})$ be a smooth section of $ L_{\Phi}$, i.e.~$[\Phi\wedge\gamma]=0$. Then $d_A\gamma$ is again a section of $L_{\Phi}$. Indeed, since $\bar\partial_A\Phi=0$ and $\partial_A\Phi=0$ by degree reasons, hence $d_A\Phi=0$, it follows that
\begin{equation*}
[\Phi\wedge d_A\gamma]=-d_A[\Phi\wedge\gamma]+[d_A\Phi\wedge\gamma]=0,
\end{equation*}
as asserted.
\end{proof}

In view of Proposition \ref{prop:linebundlepar} it follows that the unitary connection $A$ induces a connection (also denoted by $A$) on the line bundle $L_{\Phi}\to\Sigma_R^{\times}$ with associated operator $d_A=\partial_A+\bar\partial_A$ acting on $L_{\Phi}$-valued differential forms. Furthermore, from Lemma \ref{lem:indexlimtorus}  and  Proposition \ref{prop:isomorlinebundles}   we deduce that for $R>0$ the Fredholm index of the elliptic operator 
\begin{equation}\label{eq:delbaroperator}
\bar\partial_A\colon \Omega^{1,0}(\Sigma_R^\times,L_{\Phi})\to  \Omega^{1,1}(\Sigma_R^\times,L_{\Phi})
\end{equation}
equals $6(\gamma-1)$.\\
\medskip\\
We next study the behaviour of the operator family $\bar\partial_A$ in the limit $R\searrow0$. This corresponds to the passage from a family of elliptic operators with smooth coefficients for $R>0$ to the singular limiting operator $\bar\partial_A$ on $\Sigma_0^{\times}$,  a so-called   $b$-operator (cf.~  \cite{me93}). Following \cite{ss88}, a natural domain for the limiting-operator $\bar\partial_A$ is the space of sections $\gamma\in  \Omega^{1,0}(\Sigma_0^\times,L_{\Phi})$ which take the following form with respect to the  local holomorphic coordinates $z=re^{i\theta}$ and $w=se^{i\psi}$  near   $p\in\mathfrak p$ as introduced in  \textsection \ref{subsec:limitnodedsurface}.  Namely,
\begin{equation*}
\gamma(z)=u(z)\frac{dz}{z},\qquad \textrm{and}\qquad \gamma(w)=v(w)\frac{dw}{w}
\end{equation*}
for some matrix-valued functions $u\in L^2(dr\wedge d\theta)$ and  $v\in L^2(ds\wedge d\psi)$, respectively, which satisfy the matching condition $u(0)=-v(0)$. This choice of domain reflects the symmetry assumption made in (A2).\\
\medskip\\
We now suppose that  $R\geq0$ is sufficiently small. Recalling that by a result due to Biquard and Boalch (cf.~\cite{bibo04} and also \cite[Lemma 3.1]{sw15} for a description in terms of the present setup)   every solution $(A,\Phi)$ is asymptotically close to some model solution  $(A_R^{\mod},\Phi_R^{\mod})$, where 
\begin{equation*}
\Phi_R^{\mod}=\begin{pmatrix}C&0\\0&-C\end{pmatrix}\,\frac{dz}{z}\qquad (C\neq0),
\end{equation*}
we can pass to a local holomorphic frame in which the Higgs field $\Phi$ is diagonal in some neighbourhood of $\mathfrak p$. Since by Proposition \ref{prop:linebundlepar} $L_{\Phi}$ is $A$-parallel,  the $(0,1)$-part of the connection  $A$ must also be diagonal  with respect to this frame, and therefore   $\bar\partial_A$ acts as the standard (untwisted) Dolbeault operator   $\bar\partial$ near each node.\\
\medskip\\
From here on we may appeal to the analysis of the Dolbeault operator on degenerating Riemannian surfaces which has been carried out by   Seeley and Singer  \cite{ss88}.  It relies on the following functional analytic result due to Cordes and Labrousse \cite{cola63}. Let $H_1$ and $H_2$ be Hilbert spaces, and $D_t$, $t\in\R$, be a family of closed operators $H_1\supset\operatorname{dom} D_t\to H_2$ of which we assume the following. Let $G_t\subset\operatorname{dom}(D_t)\oplus H_2$ denote the graph of $D_t$ and let $P_t\colon H_1\oplus H_2\to G_t$ be the orthogonal projection. The family $D_t$ is called {\bf{graph-continuous}} at $t_0\in\R$ if $P_t$ is norm-continuous at $t_0$.

\begin{lemma}\label{lem:graphcont}
Suppose   $D_t$ is a  family of operators, graph-continuous at $t=0$,   such that $D_0$ is Fredholm. Then for all sufficiently small $t$, the operator $D_t$ is Fredholm as well and $\ind D_t=\ind D_0$.  
\end{lemma}

\begin{proof}
For a proof we refer to \cite[\textsection 2]{ss88}.
\end{proof}

The main step now is to show graph-continuity of the $R$-dependent family of operators $\bar\partial_A$ on $\Sigma_R$ at $R=0$. This  analysis has been carried out in \cite{ss88} for the Dolbeault operator  $\bar\partial$ acting on the canonical line bundle $K_{\Sigma}\cong T_{\C}^{\ast}\Sigma$.  It consists of several  steps, the first of which being the construction of a local parametrix of $\bar\partial$ on each  `neck'   $\mathcal C^-\cup\mathcal C^+$ of $\Sigma_R$ (cf.~Figure \ref{fig:2}). This    family of local parametrices is then shown to converge in a suitable sense as $R\searrow0$ to the local parametrix on $\Sigma_0$.  These local parametrices are then glued to  a $R$-independent interior parametrix to obtain graph-continuity of the family of operators $\bar\partial$ and the Fredholm property in the limit. This scheme of proof carries over without any serious changes to the family of operators $\bar\partial_A$ considered here. We therefore conclude that the latter family is graph-continuous and the limiting operator $\bar\partial_A$ on $\Sigma_0$ is Fredholm. Thus  Lemma \ref{lem:graphcont} implies the following result.

\begin{theorem}
Let $(A,\Phi)$ be a solution of the self-duality equations  \eqref{eq:hitequ} on the Riemann surface $\Sigma_R$, where $R\geq0$. Then the  operator $\bar\partial_A$ in \eqref{eq:delbaroperator}     is a  Fredholm operator  of index $6(\gamma-1)$.
\end{theorem}

\section{Concluding remarks}\label{sec:conclremarks}

In this note we did not discuss the behaviour of the full linearized operator \eqref{eq:linoperator}   in the limit $R\searrow0$. This analysis can be carried out along similar lines, showing the stability of the Fredholm index of $D_{(A,\Phi)}$ in this limit. Such a result can then be used to show bijectivity of the gluing map which assigns to a singular solution $(A_0,\Phi_0)$ of the self-duality equations on $\Sigma_0$ a smooth solution on each nearby surface $\Sigma_R$, cf.~Theorem \ref{thm:mainthm}.\\
\medskip\\
Concerning properties of the $L^2$ metric on the family of moduli spaces, it is worthwhile to point out the difference to the situation considered  in \textsection \ref{sec:fibrationlargeHiggs}. There it turned out that restriction of the $L^2$ metric to the fibres of the Hitchin  fibration persists in the limit $t\to\infty$  and induces a flat metric on each   limiting torus  $\T_q$. This is in contrast to what we encounter in the case of degenerating Riemann surfaces, where the $L^2$ metric is not defined on the limiting tori. Indeed, any $\alpha$ in the domain of $\bar\partial_A$  satisfies  the decay condition
\begin{equation*}
\lim_{z\to p}\alpha(z)=u_{\ast}\frac{dz}{z}\qquad (p\in\mathfrak p)
\end{equation*}
for some    $u_{\ast}$, the decay being at a polynomial rate in $r=\left| z\right|$. If  $u_{\ast}\neq0$ then  $\alpha$ does not have finite $L^2$ norm since $\left| u_{\ast}/z\right|^2$ is not integrable with respect to the measure $r\, dr\wedge d\theta$. Conversely, if $u_{\ast}=0$ then $\alpha$ has finite $L^2$ norm. The subspace  of such $1$-forms $\alpha$ equals the kernel of the operator  $\bar\partial_A$ under   the so-called Atiyah-Patodi-Singer (APS) boundary conditions. Therefore the  APS index theorem permits us to determine the   Fredholm index   in this case. Using  the well-known gluing properties of the APS index, it   can explicitly be computed   from the  index of $\bar\partial_A$  on a smooth surface (where it equals $6(\gamma-1)$) and the kernel of $\bar\partial_A$ acting on cross-sections of the cylindrical ends of $\Sigma_0\setminus\mathfrak p$, cf.~\cite{mapi98}.



\end{document}